\documentclass[a4paper,12pt,draft]{amsart}
\usepackage{eucal}
\usepackage{amsmath,amsthm,amssymb}
\usepackage{amsfonts}
\usepackage{latexsym}
\usepackage[dvips]{graphics}
\usepackage{amsrefs}
\usepackage{eucal}
\usepackage{mathrsfs}
\theoremstyle{plain}
\newtheorem{thm}{Theorem}[section]
\newtheorem{prop}[thm]{Proposition}
\newtheorem{lemma}[thm]{Lemma}

\theoremstyle{definition}
\newtheorem{defn}[thm]{Definition}

\theoremstyle{remark}
\newtheorem{rem}[thm]{Remark}

\numberwithin{equation}{section}

\title{Algebraic $K$-theory of finitely generated projective modules on $\E$-rings} 
\date{\today} 
\author{Mariko Ohara}
\address{Department of Mathematical Sciences, Shinshu University, 3-1-1, Asahi, Matsumoto City, JAPAN, 390-8621. 
}  
\email{primarydecomposition@gmail.com}

\newcommand{\Z}{{\mathbb{Z}}}
\newcommand{\Hom}{\mathrm{Hom}}

\newcommand{\sSet}{\mathrm{Set}_{\Delta}}
\newcommand{\sCat}{\mathrm{Cat}_{\Delta}}

\newcommand{\Map}{\mathrm{Map}}

\newcommand{\Mod}{\mathrm{Mod}}
\newcommand{\Coh}{\mathrm{Coh}}

\newcommand{\Fun}{\mathrm{Fun}}

\newcommand{\N}{\mathrm{N}}

\newcommand{\spasce}{\operatorname{Sp}}

\newcommand{\E}{\mathbb{E}_{\infty}}

\usepackage{enumerate}
\usepackage{array}
\usepackage[all]{xy} 
\usepackage{delarray}
\ifx\undefined\bysame
\newcommand{\bysame}{\leavemode\hbox to3em{\hrulefill}\,}
\fi
\begin{document}
\thispagestyle{empty}

\maketitle

\section{Introduction}

In this paper, we study the $K$-theory on higher modules in spectral
algebraic geometry. 
We relate the $K$-theory of an $\infty$-category of
finitely generated projective modules on certain $\E$-rings with
the $K$-theory of an ordinary category of finitely generated projective
modules on ordinary rings. 
We introduce earlier studies and state the main
theorem (Theorem~\ref{main2}).

\subsection{Background of this paper}
In 1990s, Elmendorf, Kriz, Mandell and May introduced a certain symmetric
monoidal category of
spectra called $S$-modules~\cite{EKMM}. 
For an algebra object $R$ in the category of $S$-modules, they also
defined a certain model category of
spectra called $R$-modules, which we denote by $\mathcal{M}_R$. 

Let $\mathcal{E}_R$ be the full subcategory of $\mathcal{M}_R$ such that
$\mathcal{E}_R$ consists of cofibrant-fibrant $R$-modules with only finitely many
non-zero homotopy groups which are finitely generated over the $0$th
homotopy group $\pi_0 R$ of $R$. 
Blumberg and Mandell  
showed the following theorem. 
\begin{thm}[\cite{BM}]\label{bbm}
Assume further that $\pi_0 R$ is a Noetherian ring. The $K$-theory of the category $\mathcal{E}_R$ is equivalent to the
$K$-theory of the ordinary category of finitely generated $\pi_0
R$-modules. 
\end{thm}
The main theorem is an analogy of Theorem~\ref{bbm} generalized to the setting of $\infty$-categories as
follows.

\subsection{Main theorem of this paper}
Although there are a lot of languages of higher category theory, we use
the same notation in Lurie's book~\cite{HT} and paper~\cite{HA}. 

Let $\mathcal{S}_*$ be the $\infty$-category of pointed
spaces~\cite[Definition 7.2.2.1]{HT}. 
The {\it stable $\infty$-category of spectra}, $\mathrm{Sp}$, is
obtained by the stabilization of $\mathcal{S}_*$
(cf. \cite[Section 6.2.2]{HA}), which has a canonical symmetric
monoidal structure induced from the cartesian products on
$\mathcal{S}_*$~\cite[Example 6.2.4.13, 6.2.4.17]{HA}. 
We denote by $\mathbb{S}$ the sphere spectrum. 
We say that a spectrum $E$ is {\it connective} if $\pi_n E \simeq 0$ for $n <
0$. 

Let $R$ be an $\E$-ring (cf. \cite[Definition 2.1.2.7]{HA}). We also have an
$\infty$-category $\Mod_R$, which is called the {\it
$\infty$-category of $R$-module} of $\mathrm{Sp}$~\cite[Section
4.2]{HA}. 
Since the tensor product on $\spasce$ is compatible with the geometric realizations~\cite[Corollary 4.8.2.19]{HA}, $\Mod_R$ becomes the
 symmetric monoidal $\infty$-category by \cite[Theorem 4.5.2.1]{HA}. We
 denote by $\otimes_R$ the tensor product on $\Mod_R$. 
We define an $\infty$-category of perfect
 $R$-modules by the smallest stable full $\infty$-subcategory of $\Mod_R$
 which contains $R$ and is closed under retracts, and let us denote this
 $\infty$-category by $\Mod_R^{perf}$~\cite[Definition 7.2.5.1]{HA}. We say that an $R$-module $M$ in $\Mod_R$ is a perfect $R$-module if it belongs to $\Mod_R^{perf}$.

Let $R$ be a connective $\E$-ring and $M$ an $R$-module. We say that $M$ is finitely generated projective if it is a
retract of a finitely generated free $R$-module~\cite[Proposition 7.2.2.18]{HA}. 
We denote by $\Mod_R^{proj}$ the $\infty$-category of
finitely generated projective $R$-modules.  

Let $R$ be a connective $\E$-ring with only finitely many non-zero
homotopy groups. 
Let $\mathcal{P}_{\pi_0 R}$ be an ordinary category of
 finitely generated projective $\pi_0 R$-modules. Barwick and Lawson~\cite{BL} showed that, if $R$ is a regular $\E$-ring with only finitely many non-zero
homotopy groups, $K(\Mod_R^{perf})$ is equivalent to the $K$-theory of
the ordinary category $\mathcal{P}_{\pi_0 R}$. Here we recall the notion of regularity on $R$ in
Definition~\ref{regreg}. 

We obtain the following theorem for $K(\Mod_R^{proj})$. The statement is
an analogy of the result of Barwick and Lawson for $K(\Mod_R^{perf})$. 

\begin{thm}[cf. Theorem~\ref{m2}]
\label{main2}
Let $R$ be a regular $\E$-ring with only finitely many non-zero
homotopy groups. 
Then, there is a weak equivalence $K(\Mod_{R}^{proj}) \simeq K(\mathcal{P}_{\pi_0 R})$. 
\end{thm}

\subsection{Remarks for Theorem~\ref{main2}}
As a key lemma for Theorem~\ref{main2}, we show the equivalence
$K(\Mod_R^{proj}) \simeq K(\Mod_R^{perf})$.  
The difficulty in showing the equivalence is that the resolution theorem is not established in the $K$-theory of $\infty$-categories. 

Recently, Mochizuki~\cite{Moc1} proved a resolution theorem in
Waldhausen $K$-theory with certain assumptions which is explained in Theorem~\ref{moc}. We construct the
relation between a sequence of certain subcategories in $\mathcal{M}_R$ and
 that in $\Mod_R^{perf}$ respectively, and apply his resolution theorem to the
sequence of subcategories in $\mathcal{M}_R$. 

We remark that
Lurie also shows the equivalence $K(\Mod_R^{proj}) \simeq
K(\Mod_R^{perf})$ in the completely $\infty$-categorical setting for a connective $\E$-ring $R$ in his
lecture~\cite{Lurielec19}. 
It is a different kind of proof from ours.  

\subsection{Outline of this paper}
This paper is organized as follows. 
In Section $2$, we introduce the
terminology of relative categories and Dwyer-Kan localization, and show
Lemma~\ref{nn1} together
with the functorial factorization and homotopically full condition. 
In Section $3$, we have the
existence of mapping cylinders and mapping path spaces in the categories
of spectra which we mainly use in this paper. We demonstrate the
compatibility of simplicialicity and monoidalicity on
the category of symmetric spectra, and have the correspondence between the
category of $R$-modules in the sense of Elmendorf-Kriz-Mandell-May and the
$\infty$-category of $R$-modules as Proposition~\ref{sm1}. 
In Section $4$, we define the notion of perfect $R$-modules in the
category of $R$-modules in the sense of Elmendorf-Kriz-Mandell-May, and
prove the correspondence in Lemma~\ref{k86} between the perfect $R$-modules in the sense of Elmendorf-Kriz-Mandell-May and
the $\infty$-categorical perfect $R$-modules. By using these
consequences, we show an equivalence between the algebraic $K$-theory of
perfect $R$-modules and the algebraic $K$-theory of finitely generated
projective $R$-modules in Proposition~\ref{perfproj} in Section $5$. 
In Section $6$, by using the consequence of Barwick and Lawson, we prove Theorem~\ref{m2}. 

\subsubsection*{Acknowledgement}
The author would like to express deeply her thanks to Professor Nobuo Tsuzuki
for his valuable advice and checking this paper. The author would like
to express her thanks to Professor Satoshi Mochizuki who devoted many
hours to proofreading my draft. He also suggested a lot of helpful 
comments, especially the resolution theorem and the cofinality. His
valuable comments improved the depth of my understanding.

\section{Relative categories and Dwyer-Kan localization}
A relative category is a pair $(\mathcal{C}, W)$ consisting of a
category $\mathcal{C}$ and a subcategory $W \subset \mathcal{C}$ whose objects
are the objects of $\mathcal{C}$ and whose class of morphisms is a certain class of morphisms in
$\mathcal{C}$. The morphisms in $W$ are called weak equivalences. 
For relative categories $(\mathcal{C}, W)$ and
$(\mathcal{C}', \, W')$, a functor $(\mathcal{C}, W) \to (\mathcal{C}',
\, W')$ of relative categories is defined by a functor $\mathcal{C} \to
\mathcal{C}'$ which sends $W$ in $W'$.

Recall that a monoidal category is an ordinary category $\mathcal{C}$
equipped with an associative product $\otimes : \mathcal{C} \times \mathcal{C}
\to \mathcal{C}$ and a unit object $1$~\cite[A. 1.3]{HT}.

Let $\sSet$ denote the category of simplicial sets. In this paper, we use the term ``a simplicial category'' as a
 $\sSet$-enriched category. (For the definition of enrichment, see
 \cite[A. 1.4]{HT}.) 
Let $\sCat$ denote the category of simplicial categories such that the
objects are small simplicial categories, and a morphism is a $\sSet$-enriched functor.

We assume that $\sCat$ is endowed with the Bergner model
 structure (cf. \cite[Definition A.3.2.16]{HT}) and call these weak
 equivalences Dwyer-Kan equivalences. The study of Bergner proved that a
 fibrant object of this model category is a fibrant simplicial category
 i.e. a simplicial category whose mapping spaces are Kan complexes (for a proof, see \cite[Theorem A 3.2.2.4]{HT}).

There is a
simplicial nerve functor $\N_\Delta:  \sCat \to \sSet$ which is the right Quillen functor of model categories,
\[
 \mathfrak{C} : \sSet \rightleftarrows \sCat : \N_{\Delta},
\]
where the model structure on $\sSet$ is Joyal model structure and that
on $\sCat$ is Bergner model structure. 
They induce the
Quillen equivalence~\cite[Theorem 2.2.5.1]{HT}.  

\begin{defn}[Relative nerve \cite{HA} Definition 1.3.4.1]\label{reln}
Let $\mathcal{C}$ be an $\infty$-category, and $W$ a collection of morphisms in $\mathcal{C}$. 
\begin{enumerate}[(i)]
\item We say that a morphism $f : \mathcal{C} \to \mathcal{D}$ exhibits $\mathcal{D}$ as an $\infty$-category obtained from $\mathcal{C}$ by inverting the set
 of morphisms $W$ if, for every $\infty$-category $\mathcal{E}$,
 the composition with $f$ induces a fully faithful embedding
 $\Fun(\mathcal{D}, \mathcal{E}) \to \Fun(\mathcal{C}, \mathcal{E})$ whose
essential image is the collection of functors $\mathcal{C} \to \mathcal{E}$ which carry each morphism in $W$ to an equivalence
in $\mathcal{E}$. 
\item In the case of (i), the $\infty$-category $\mathcal{D}$ is
 determined uniquely up to equivalence by $\mathcal{C}$ and $W$. If $\mathcal{C}$ is an ordinary category and $W$ is a collection of
 morphisms in $\mathcal{C}$, we proceed the above construction for
 $\N_{\Delta}(\mathcal{C})$ and denote the $\infty$-category
 $\N_{\Delta}(\mathcal{C})[W^{-1}]$. 
\end{enumerate}
We call $\N_{\Delta}(\mathcal{C})[W^{-1}]$ the relative nerve of
 ($\mathcal{C}$, $W$). 
\end{defn}

\begin{rem}
For a relative category $(\mathcal{C}, W)$, the condition that the
 subcategory $W$ contains the every object in $\mathcal{C}$ implies the condition that the morphisms in $\N_{\Delta}(\mathcal{C})$ spanned by $W$ contains all the degenerate edges.  
\end{rem}


Let $(\mathcal{C}, W)$ be a relative category. Let
 $W_{\N_{\Delta}(\mathcal{C})}$ is a collection of morphisms which is
 generated by the image of $W$ in $\N_{\Delta}(\mathcal{C})$. If we
 regard $\N_{\Delta}(\mathcal{C})[W_{\N_{\Delta}(\mathcal{C})}^{-1}]$ as
 a marked simplicial set
 $\N_{\Delta}(\mathcal{C})[W_{\N_{\Delta}(\mathcal{C})}^{-1}]^{\natural}$,
 then it is a fibrant replacement of $(\N_{\Delta}(\mathcal{C}),
 W_{\N_{\Delta}(\mathcal{C})}))$ in $(\sSet^+ )_{/\ast}$~\cite[Remark
 1.3.4.2]{HA}. It follows from that the morphism $\N_{\Delta}(\mathcal{C}) \to \ast$ factors through
 $\N_{\Delta}(\mathcal{C})[W_{\N_{\Delta}(\mathcal{C})}^{-1}] \to \ast$
 by the definition of relative nerve.

\subsection{Localization of relative categories}
%
%
%

Recall that, for a relative category $(\mathcal{C}, W)$, there is a
 construction of a simplicial category
from a relative category, which is called a Dwyer-Kan localization~\cite{DK2}. 


A hammock localization~\cite[Section 2.1]{DK2} of $(\mathcal{C}, W)$ is one of the explicit construction of a simplicial category
from a relative category~\cite[Proposition 2.2]{DK2}, \cite{DK1}. 
For an ordinary relative category $(\mathcal{C}, W)$, we denote by $L^H
 (\mathcal{C}, W)$ its hammock localization. 

Recall that the objects of $L^H (\mathcal{C}, W)$ are the objects of
 $\mathcal{C}$. 
For $X, Y \in \mathcal{C}$, we denote by $L^H
 (\mathcal{C}, W)(X, \, Y)$ the mapping space of $X, Y \in L^H
 (\mathcal{C}, W)$. Note that $L^H(\mathcal{C}, W) (-,-)$ satisfies
 the axiom of identity and associative
composition of enrichment~\cite[A 1.4]{HT} by concatenating the horizontal
morphisms in the hammock diagrams. .

Note that the homotopy category $hL^H(\mathcal{C}, W)$ of the hammock
localization $L^H(\mathcal{C}, W)$ is categorically equivalent to $\mathcal{C}[W^{-1}]$~\cite[Proposition 3.1]{DK2}. 


For a relative category $(\mathcal{C}, W)$, let $L^H(\mathcal{C},
 W))^{fib}$ be the hammock localization, where $(-)^{fib}$
denotes a fibrant replacement with respect to the Bergner model structure. Then, we have an equivalence 
$\N_{\Delta}((\mathcal{C}))[W^{-1}] \simeq \N_{\Delta}((L^H (\mathcal{C}, W))^{fib})$ of $\infty$-categories~\cite[Proposition 1.2.1]{hinich}. 


For the terminology of model category in this paper is as follows.  

In terminology of Lurie, which we adopt in this paper,
      the model categories are required the existence of small limits
      and colimits. (On the other hand, Quillen~\cite{Qui} required the
      existence of only finite limits and colimits). However, since the model
      categories which will appear in this paper have small limits and
      colimits, we use the notation ``a model category'' in the sense of
      Lurie in this paper. 

We remark that the terminology of
Quillen~\cite{Qui} is used in Elmendorf-Kriz-Mandell-May~
\cite{EKMM}, Mandell-May-Schwede-Shipley~\cite{MMSS},
Dwyer-Kan~\cite{DK1} \cite{DK2} \cite{DK3} and
Barwick-Kan~\cite{BK4} \cite{BK3}. 

For a model category $\mathcal{C}$, let $\mathcal{C}^c$, $\mathcal{C}^f$
and $\mathcal{C}^{\circ}$ be the subcategory of cofibrant, fibrant and
cofibrant-fibrant objects respectively. We denote by $W_{\mathcal{C}}$
the subcategory of weak equivalences in $\mathcal{C}$.  

\begin{defn}[\cite{BK3} Section 1.2]\label{hoful}
For a model category $\mathcal{C}$, let $(\mathcal{C}, W_{\mathcal{C}})$ be a
 relative category arising from $\mathcal{C}$. For a relative subcategory
 $(\mathcal{D},  W_{\mathcal{D}})$ of $(\mathcal{C}, W_{\mathcal{C}})$, we say that $(\mathcal{D}, W_{\mathcal{D}})$ is a homotopically full
 relative subcategory if it is a relative category of the form
 $(\mathcal{D}, W_{\mathcal{C}} \cap \mathcal{D})$, where $\mathcal{D}$ is a full
 subcategory of $\mathcal{C}$, which has the following property: if an
 object 
$C \in \mathcal{C}$
 has a zig-zag of weak equivalences to an object $D \in \mathcal{D}$,
 then $C \in \mathcal{D}$. 
\end{defn}

\begin{defn}[\cite{BM1} Example 5.2]\label{52}
Let $(\mathcal{C}, W)$ be a relative category, $A$ and $B$ objects of $\mathcal{C}$. We define a category $W^{-1}\mathcal{C}(A, \, B)$ by the following data.
\begin{enumerate}[(i)]
\item An object is a pair $\{X, f_1, f_2\}$, where $X \in \mathcal{C}$, $f_1 : B \to X$ is a morphism in $W$ and $f_2 : A \to X$ is a morphism in $\mathcal{C}$. 
\item A morphism $\{X, f_1, f_2 \} \to \{X' f'_1, f'_2\}$ is defined as the following diagram
\[ 
 \xymatrix@1{
 A \ar[r]^{f_2} \ar[d]^{f'_2} & X \ar[dl]^{\phi} \\
  X'  & B \ar[u]_{f_1} \ar[l]^{f'_1} ,
} 
\]
where $\phi : X \to X'$ is a morphism in $W$. 
\end{enumerate}
\end{defn}
In the case of homotopically full subcategories of model categories
whose subcategory of cofibrants admits the functorial
factorization~\cite[Section 2]{BM1}, we will describe the mapping space $L^H(-,-)$ more simply. 
In general, a combinatorial model category admits the functorial
cofibrant and fibrant replacement. Therefore, its subcategory of
cofibrants and its opposite of subcategory of fibrants especially admit the functorial factorization. 
Note that, in terminology of Quillen, a model category is not assumed to
have the functorial cofibrant and fibrant replacement. 

\begin{lemma}\label{nn1}
Let $\mathcal{C}$ be a model category whose subcategory of
cofibrants admits the functorial factorization in the sense of
 \cite[Section 2]{BM1}.  
Let $(\mathcal{C}, W_{\mathcal{C}})$ be a
 relative category arising from $\mathcal{C}$, and $(\mathcal{D},  W_{\mathcal{D}})$
 a homotopically full relative subcategory of $(\mathcal{C}, W_{\mathcal{C}})$. 
\begin{enumerate}[(i)]
\item If $A$ and $B$ are cofibrant-fibrant objects in $\mathcal{D}$ (resp. $\mathcal{C}$),
      $\N_{\Delta}(W_{\mathcal{D}^c}^{-1}\mathcal{D}^c(A, B)) \simeq L^H(\mathcal{D},
      W_{\mathcal{D}})(A, B)$
      (resp. $\N_{\Delta}(W_{\mathcal{C}^c}^{-1}\mathcal{C}^c(A, B)) \simeq L^H(\mathcal{C},
      W_{\mathcal{C}})(A, B)$). Here, we regard a category $W_{\mathcal{D}^c}^{-1}\mathcal{D}^c(A,
      B)$ (resp. $W_{\mathcal{C}^c}^{-1}\mathcal{C}^c(A, B)$) as a trivial simplicial category
\item We have $L^H (\mathcal{D}, W_{\mathcal{C}} \cap \mathcal{D})(A, B)
      \simeq L^H (\mathcal{C}, W_{\mathcal{C}} )(A, B)$ for cofibrant-fibrant objects $A, B \in \mathcal{D}$, . 
\end{enumerate}
\end{lemma}
\begin{proof}
We prove (i). 
By the assumption, $\mathcal{C}^c$
 satisfies the assumption of \cite[Proposition 5.4]{BM1} by
 \cite[Section 2.2]{BGT-End}. Therefore, we apply \cite[Proposition
 5.4]{BM1} to $\mathcal{C}^c$. 
Since $\mathcal{D}$ is the homotopically full
 subcategory of $\mathcal{C}$, we also apply \cite[Proposition
 5.4]{BM1} to $\mathcal{D}^c$, so that we have $\N_{\Delta}(W_{\mathcal{C}^c}^{-1}\mathcal{C}^c(-, -))
 \simeq L^H(\mathcal{C}^c, W_{\mathcal{C}^c})(-,-)$ and $\N_{\Delta}(W_{\mathcal{D}^c}^{-1}\mathcal{D}^c(-, -)) \simeq L^H(\mathcal{D}^c, W_{\mathcal{D}^c})(-,-)$. By \cite[Proposition 5.2]{DK3}, we have
 $L^H(\mathcal{C}^c, W_{\mathcal{C}^c}) \simeq L^H(\mathcal{C}, W_{\mathcal{C}})$ and
 $L^H(\mathcal{D}^c, W_{\mathcal{D}^c}) \simeq L^H(\mathcal{D}, W_{\mathcal{D}})$.

For (ii), let $A$ and $B$ be objects in $\mathcal{D}$. 
By comparing the diagram $W_{\mathcal{D}^c}^{-1}\mathcal{D}^c(A,\, B)$ with $W_{\mathcal{C}^c}^{-1}\mathcal{C}^c(A,\, B)$, we have
 the weak homotopy equivalence $W_{\mathcal{D}^c}^{-1}\mathcal{D}^c(A,\, B) \simeq
 W_{\mathcal{C}^c}^{-1}\mathcal{C}^c(A,\, B)$. By applying (i), we have $L^H (\mathcal{D}, W_{\mathcal{C}} \cap \mathcal{D})(A, B) \simeq L^H
 (\mathcal{C}, W_{\mathcal{C}} )(A, B)$. 
\end{proof}

\section{Relation between classical and $\infty$-categorical spectra}
\subsection{$R$-modules in the sense of Elmendorf-Kriz-Mandell-May}

Throughout this paper, we denote by $\wedge$ the smash products in
the sense of topology (defined on spaces and spectra). 

Let $R$ be an $S$-algebra, and $\mathcal{M}_R$ the category of $R$-modules in the sense of
Elmendorf-Kriz-Mandell-May~\cite{EKMM}. Recall that $\mathcal{M}_R$
admits the symmetric monoidal structure~\cite[II, Section 3]{EKMM} by
the smash product over $R$, which we denote by $\wedge_R$. 

The category $\mathcal{M}_R$ admits a model structure which is defined in ~\cite[VII Section 4]{EKMM}. 
Let us denote the subcategory of weak equivalences in $\mathcal{M}_R$ by
 $W_{\mathcal{M}_R}$. 
\begin{rem}[cf. \cite{SS}]\label{ms1}
Note that every object in $\mathcal{M}_R$ is fibrant. Therefore, we have
 $\mathcal{M}_R^f = \mathcal{M}_R$ and $\mathcal{M}_R^c = \mathcal{M}_R^{\circ}$. 
\end{rem}
Let $I$ denote the unit interval $[0, 1]$, where we regard $\{0\}$ as
the base point. 

\begin{defn}\label{mcyl1}
Let $A$ and $B$ be objects in $\mathcal{M}_R$.  
\begin{enumerate}[(i)]
\item A mapping cylinder for $f: A \to B$, denoted by $Mf$, is defined by the following pushout:
\[ 
 \xymatrix@1{
 A \ar[r]^f \ar[d]^{\partial_1} & B \ar[d]^{j_f} \\
  A \wedge I \ar[r]^{\pi_f} & Mf,
} 
\]
where $A \wedge I$ is a cylinder. 
\item Dually, a mapping path space for $f: A \to B$, denoted by $Nf$, is defined by the following pullback:
\[ 
 \xymatrix@1{
 Nf \ar[r]^{\pi'_f} \ar[d]^{j'_f} & B^I \ar[d]^{d_1} \\
  A \ar[r]^f & B. 
} 
\]
\end{enumerate}
\end{defn}
\begin{lemma}\label{Mcyl}
Mapping cylinders and mapping path spaces exist in the model category $\mathcal{M}_R$. 
\end{lemma}
\begin{proof}
By \cite[III, Lemma 3.2]{EKMM}, the tensor product on $\mathcal{M}_R$ has the right adjoint, so that $\mathcal{M}_R$ has a path space. Then, the assertion follows from \cite[III, Theorem 1.1]{EKMM}. 
\end{proof}


Let $A$ and $B$ be cofibrant objects in $\mathcal{M}_R$. Let $f: A \to B$ be a morphism. 
A mapping cylinder gives a functorial factorization for $\mathcal{M}_R^c$, i.e., every map $f : A
      \to B$ factors through $A \to Mf \to B$, where the map $A \to Mf$
      is a cofibration and $Mf \to B$ has a natural section $B \to Mf$,
      which is a weak equivalence. Moreover, $Mf$ is a cofibrant
      object. 

Dually, a mapping path object gives a functorial factorization for the
opposite category $(\mathcal{M}_R)^{op}$ (Recall that every object in
      $\mathcal{M}_R$ is fibrant.), i.e., every map $f : A \to B$ factors through $A \to Nf \to B$ of the map $A \to Nf$ has a natural projection $Nf \to A$, which is a weak equivalence, and $Nf \to B$ is a fibration. 

\begin{defn}\label{hrw}
Let $R$ be a commutative $S$-algebra, $A$ and $X$ $R$-modules in the
 sence of Elmendorf-Kriz-Mandell-May. 
Let $i: A \to X$ a morphism of
$R$-modules. We say that $i$
is a Hurewicz cofibration if $Mi = X \cup_i (A
\wedge I)$ has the cylinder $X \wedge I$ as a retract. 
Here, a mapping cylinder is defined in Definition~\ref{mcyl1}. 
\end{defn}
\begin{rem}\label{Hcf}
Note that, by \cite[VII, Theorem 4.15]{EKMM}, the Hurewicz cofibrations
 include the cofibrations in $\mathcal{M}_R$. 
Especially, on the cofibrant-fibrant objects, the weak equivalences are
 homotopy equivalences and the Serre fibrations are Hurewicz fibrations,
 so that the Hurewicz cofibrations is the cofibrations. 
\end{rem}

\subsection{The category of symmetric spectra}
According to \cite{HSS}, we recall the notion of symmetric spectra
as follows. 

Let $\mathcal{M}^{sym}$ be a category of symmetric spectra built from
the simplicial sets~\cite[Section 2.2]{HSS}. 
We regard the category $\mathcal{M}^{sym}$ as a simplicial category by
      \cite[Definition 2.1.10]{HSS}. The category $\mathcal{M}^{sym}$ is a symmetric monoidal category with
respect to tensor products~\cite[Definition 2.1.3]{HSS} which is induced
      from the smash product. A (commutative) algebra object $R \in \mathcal{M}^{sym}$ with respect to
      the tensor product on $\mathcal{M}^{sym}$ is called a symmetric
      (commutative) ring spectrum. 

For a symmetric ring spectrum $R$, an $R$-module object in
      $\mathcal{M}^{sym}$ with respect to the tensor product on $\mathcal{M}^{sym}$ is called a
      symmetric $R$-module spectrum. Let $\mathcal{M}_R^{sym}$ be a category which consists symmetric $R$-module spectra in $\mathcal{M}^{sym}$ and the morphisms compatible with the $R$-module structure.   
We also denote by $\wedge_R$ the smash product
 on $\mathcal{M}_R^{sym}$.

There exists a left proper combinatorial model structure on
 $\mathcal{M}^{sym}$ defined in \cite[Definition 9.1]{MMSS}, which is
 called the stable model structure.

\subsection{Compatibility of simplicialicity and monoidalicity on $\mathcal{M}^{sym}$}
We say that $\mathcal{A}$ is a simplicial model category if it is a
 $\sSet$-enriched model category in the sense of \cite[Definition A.3.1.5]{HT}. 

\begin{rem}[\cite{HA}, Example 4.1.3.6]\label{tensym}
Let $\mathcal{M}$ be an ordinaly symmetric monoidal category. Then, the
 $\infty$-category $\N_{\Delta}(\mathcal{M}^c)[W^{-1}]$ 
inherits the symmetric monoidal structure.  
\end{rem}
The following definition is a special case of \cite[Definition
4.1.3.7]{HA}. 
\begin{defn}
\label{4137}
Let $\mathcal{C}$ be a simplicial category. 
\begin{enumerate}[(i)]
\item We say that a (symmetric) monoidal structure on $\mathcal{C}$ is
      weakly compatible with the simplicial structure on $\mathcal{C}$
      if the (symmetric) monoidal operation $\otimes : \mathcal{C} \times \mathcal{C} \to \mathcal{C}$ is a simplicial functor, and the associativity and unit condition (and symmetricity) is given by a natural transformation of simplicial functors. 
\item Furthermore, assume that the monoidal structure on $\mathcal{C}$ satisfies the following conditions: for every $X, Y \in \mathcal{C}$, there exists an object $X^Y \in \mathcal{C}$ and evaluation morphism $e: X^Y \otimes Y \to X$ such that the morphism induces a bijection $\Hom_{\mathcal{C}}(Z, \, X^Y) \cong \Hom_{\mathcal{C}}(Z \otimes Y, \, X)$ for any $Z \in \mathcal{C}$. 
\item We say that a symmetric monoidal structure on $\mathcal{C}$ is
 compatible with the simplicial structure on $\mathcal{C}$ if it is
 weakly compatible and the evaluation morphism $e$ induces an
 isomorphism $\Map_{\mathcal{C}}(Z, \, X^Y) \cong \Map_{\mathcal{C}}(Z
 \otimes Y, \, X)$ of simplicial sets for any $Z \in \mathcal{C}$. 
\end{enumerate}
\end{defn}
\begin{defn}[cf. \cite{HA} Definition 4.1.3.8]\label{4138}
A simplicial symmetric monoidal model category is a symmetric monoidal model category which is also equipped with the structure of a simplicial model category such that the simplicial structure and the (symmetric) monoidal structure are compatible in the sense of Definition~\ref{4137}. 
\end{defn}

Now, let $\mathcal{M}^{sym}$ be the category of symmetric spectra
endowed with the stable model structure. 
\begin{prop}\label{symsym}
\begin{enumerate}[(i)]
\item The category $\mathcal{M}^{sym}$ is a simplicial model category. 
\item The category $\mathcal{M}^{sym}$ is a symmetric monoidal model category.
\end{enumerate}
\end{prop}
\begin{proof}
To prove that $\mathcal{M}^{sym}$ is symmetric monoidal model category, we check the conditions in \cite[Definition A.3.1.2]{HT}. The condition (i) follows from \cite[Proposition 3.4.2 (4), Theorem 5.3.7 (5)]{HSS}. The symmetric sphere is a cell, so that it is cofibrant. The symmetric monoidal structure of $\mathcal{M}^{sym}$ is closed by \cite[Theorem 5.5.2]{HSS}. 

To prove that $\mathcal{M}^{sym}$ is a simplicial model category, what we
 prove is the conditions in \cite[Definition A.3.1.1]{HT} and \cite[Definition A.3.1.5]{HT}. 

The condition (i) of \cite[Definition A.3.1.5]{HT} follows from
 \cite[Proposition 1.3.1]{HSS}. For the condition (ii) of \cite[Definition A.3.1.5]{HT}, we check the
 conditions in \cite[Definition A.3.1.1]{HT}. The condition (i) of \cite[Definition A.3.1.1]{HT} follows from \cite[Proposition 3.4.2 (4), Theorem 5.3.7 (5)]{HSS}. The
 condition (ii) of \cite[Definition A.3.1.1]{HT} follows from \cite[Proposition
 1.2.10]{HSS}. 
\end{proof}

\begin{prop}\label{symsym2}
$\mathcal{M}^{sym}$ is a simplicial symmetric monoidal model category in the sense of Definition~\ref{4138}
\end{prop}

\begin{proof}
We check the compatibility in Definition~\ref{4137}(ii). 
The assumption of Definition~\ref{4137} is satisfied by the definition of symmetric monoidal structure in $\mathcal{M}^{sym}$~\cite[Definition 2.2.1]{HSS}. 
By \cite[Theorem 2.1.11]{HSS}, the condition Definition~\ref{4137}(ii) holds. 
\end{proof}

\subsection{Relation between classical and $\infty$-categorical spectra}
Let $Kan_{\ast}$ be a category of pointed Kan complexes. 
Here, a pointed simplicial set is a simplicial set together with a
morphism $\ast \to X$ from a point. 
Let $W_{Kan}$ be the category of
weak homotopy equivalences in $Kan_{\ast}$. 
We can identified
$\mathcal{S}_{\ast}$ with the relative nerve $n: \N_{\Delta}(Kan_{\ast})
\to \N_{\Delta}(Kan_{\ast})[W_{Kan}^{-1}]$, where we regard $Kan_{\ast}$
as an ordinary category and $n$ is induced from the inclusion which
becomes the functor associated with the
relative nerve in Definition~\ref{reln}. 
Then, we have the following functor induced from $n$
\[
 n' : Kan_{\ast} \to \mathcal{S}_{\ast}. 
\]
Let $\Omega = \Map(S^1, -)$ be the endofunctor on $Kan_{\ast}$.  
Note that the endofunctor $\Omega = \Map(S^1, -)$ on $Kan_{\ast}$ induces the
derived functor on $\N_{\Delta}(Kan_{\ast})$ defined by homotopy
cartesin. 
Since $\N_{\Delta}$ is the
right adjoint, it preserves small limits. 
Therefore, $n'$ commutes with $\Omega$. 

We have the following commutative diagram 
\begin{equation}\label{kansp}
 \xymatrix@1{
 \cdots \ar[r]^{\Omega}  & \mathcal{S}_{\ast} \ar[r]^{\Omega} & \mathcal{S}_{\ast} \\
  \cdots \ar[r]^{\Omega}  & \ar[u]^{n'} Kan_{\ast}
 \ar[r]^{\Omega} & \ar[u]^{n'} Kan_{\ast}, 
} 
\end{equation}
where $\mathcal{S}_{\ast}$ is an $\infty$-category of pointed spaces.

The bottom sequence in the diagram (\ref{kansp}) gives the classical
spectra built from the simplicial set and the upper sequence in the
diagram gives the spectrum objects determined. The functor $n'$ gives
an assignment between them. 

Consequently, this assignment gives rise to the following equivalence in
the proposition which is a special case of \cite[Proposition B.3]{ABG}. 
\begin{prop}\label{smod1}
Let $\mathcal{M}^{sym}$ be a category of symmetric spectra endowed with
 the stable model structure.  
Then, the assignment in (\ref{kansp}) induces an equivalence of $\infty$-categories
\[
 \N_{\Delta}((\mathcal{M}^{sym})^{c})[W_{\mathcal{M}^{sym}}^{-1}] \simeq \spasce. 
\]
\end{prop}
\begin{proof}
Since the category $\sSet$ endowed with the Kan model structure is a left proper celluler simplicial model category, we can apply \cite[Proposition B.3]{ABG} to $\mathcal{M}^{sym}$. Take $\mathcal{C} = \sSet$. Then,
 the left hand side is a model category of symmetric spectra endowed
 with the stable model category and the right hand side is the
 stabilization of the $\infty$-category $\mathcal{S}$, which is
 equivalent to $\spasce$. 
\end{proof}

By Proposition~\ref{symsym} and Proposition~\ref{symsym2},
$\mathcal{M}^{sym}$ satisfies the assumption \cite[Proposition
4.3.3.15]{HA}, so that we have the
following proposition. 

\begin{prop}
\label{43315}
Let $R$ be a commutative cofibrant symmetric ring spectrum and
$\mathcal{M}_R^{sym}$ the category of symmetric $R$-module spectra. 

Then, there is a combinatorial model structure on $\mathcal{M}_R^{sym}$
 defined as follows: 
\begin{enumerate}[(i)]
\item A morphism of $R$-modules is a weak equivalence
 (resp. a fibration) if it is a weak equivalence (resp. a fibration) as
 a morphism of spectra in $\mathcal{M}^{sym}$ endowed with the stable
 model structure. 
\item If we regard $\mathcal{M}^{sym}$ as simplicial monoidal model category, $\mathcal{M}_R^{sym}$ becomes a simplicial model category. 
\end{enumerate}
In terminology of \cite[Theorem 12.1]{MMSS} and \cite[Corollary
5.4.3]{HSS}, the model structure on $\mathcal{M}_R^{sym}$ which is
defined in Proposition~\ref{43315} is called the stable model
structure. We denote by $W_{\mathcal{M}_R^{sym}}$ the subcategory of weak
equivalences. 
\end{prop}
\qed

Since an $\infty$-category of $(A, \, A)$-bimodules is equivalent to
 $\Mod_A$, we have the following proposition by applying 
Proposition~\ref{43315} and \cite[Theorem 4.3.3.17]{HA} to
 $\mathcal{M}^{sym}$ endowed with the stable model structure.

\begin{prop}[\cite{HA}, Theorem 4.3.3.17]\label{smod3}
Let $R$ be a commutative cofibrant symmetric ring spectrum and
$\mathcal{M}_R^{sym}$ the category of symmetric $R$-module spectra
 endowed with the stable model structure. 
There is an equivalence of $\infty$-categories
\[
 \N_{\Delta}((\mathcal{M}_R^{sym})^{c})[W_{\mathcal{M}_R^{sym}}^{-1}] \simeq \Mod_{R'},
\]
where $R'$ is an object in $\spasce$ corresponding to $R$ under the
 equivalence in Proposition~\ref{smod1}, which becomes an $\E$-ring. 
\end{prop}
\qed

We introduce another model structure on $\mathcal{M}_R^{sym}$. 

Let $R$ be a symmetric ring spectrum. Let $\mathcal{M}_R^{sym}$ be the
category of symmetric $R$-module. Assume that $\mathcal{M}_R^{sym}$ is
endowed with the positive model structure which is defined in \cite[Section
14]{MMSS}. By \cite[Proposition 14.6]{MMSS}, the positive model structure is Quillen equivalent to the
stable model structure obtained in
Proposition~\ref{43315}. 

Since $\mathcal{M}_R^{sym}$ is built via the sequences of simplicial
 sets and has the set of generating cofibrations and acyclic cofibrations by
\cite[Theorem 14.1]{MMSS}, it is a combinatorial model
category for suitable cardinal~\cite[Proposition 3.2.3.13]{HSS}. Therefore, we can take a cofibrant replacement
 functorially~\cite[Proposition 1.2.5]{HT}. 


Let $A$ be an $S$-module in the sense of Elmendorf-Kriz-Mandell-May. 
We say that a cofibrant $S$-module $A^{-1}$ is a cofibrant desuspension
of $A$ if $A^{-1}$ is endowed with a weak equivalence $A^{-1} \wedge_S
S^1 \to A$, where $\wedge_S$ is the smash product over $S$. 
Let $\mathcal{M}_S$ be the category of $S$-modules. 
By virtue of \cite[II, 1.7]{EKMM}, there exists a cofibrant desuspension of $\mathbb{S}$ in $\mathcal{M}_S$. 
We denote by $(\mathbb{S}^{-1})^{c}$ a cofibrant desuspension of $\mathbb{S}$.  
Set $(\mathbb{S}^0)^c = \mathbb{S}$. We define a
 functor $\Phi : \mathcal{M}_S \to \mathcal{M}^{sym}$ by sending $M$ to
 a symmetric spectra $\Phi(M)$ whose $n$-th space is given by
\[
 \Phi(M)_n = \mathcal{M}_S(((\mathbb{S}^{-1})^n)^{c}, \, M). 
\]

It was proved in \cite{schwede} that $M$ and $\Phi(M)$ have the same homotopy groups. 

Schwede~\cite{schwede} constructed a Quillen equivalence between the
spectra, algebras and modules in $\mathcal{M}^{sym}$ and
$\mathcal{M}_S$.

\begin{thm}[\cite{schwede}]\label{schw}
Let $R$ be a cofibrant-fibrant commutative $S$-algebra. Let $Q$ be a cofibrant
 replacement functor on $\mathcal{M}_{\Phi(R)}^{sym}$ \cite[Proposition
 3.2.3.13]{HSS}. 
The functor $\Phi$ has a left adjoint denoted by $\Lambda$, and they
 induce a Quillen equivalence 
\[
 \mathcal{M}_{Q \Phi(R)}^{sym} \rightleftarrows \mathcal{M}_R,
\]
where $\mathcal{M}_R$ is endowed with the model structure and $\mathcal{M}_{Q \Phi(R)}^{sym}$ is endowed with the
 positive stable model structure. 
\end{thm}
\qed

Recall that the class of weak equivalences in $\mathcal{M}_R^{sym}$ with
 respect to stable model structure (resp. with respect to
 positive stable model structure) is denoted by $W_{\mathcal{M}_R^{sym}}$
 (resp. $W'_{\mathcal{M}_R^{sym}}$) and the class of weak equivalences in $\mathcal{M}_R$ is denoted by $W_{\mathcal{M}_R}$. 
\begin{lemma}\label{smi}
There is an equivalence 
$\N_{\Delta}((\mathcal{M}_R^{sym})^c)[(W'_{\mathcal{M}^{sym}_R})^{-1}]
 \simeq \N_{\Delta}((\mathcal{M}_R^{sym})^c)[W_{\mathcal{M}^{sym}_R}^{-1}]$ of
 $\infty$-category induced by the identity functor on $\mathcal{M}_R^{sym}$, where the left hand side is obtained by the
 positive stable model structure on $\mathcal{M}_R^{sym}$ and the right
 hand side is obtained by the
 stable model structure on $\mathcal{M}_R^{sym}$. 
\end{lemma}
\begin{proof}
Since the model structure defined in Proposition~\ref{43315} (resp. with
 respect to positive model structure) is combinatorial, we apply
 \cite[Lemma 1.3.4.2.1]{HA} to the Quillen equivalence between the
 stable and positive stable model structures. Then, we obtain that the identity functor on $\mathcal{M}_R^{sym}$ induces an equivalence 
$\N_{\Delta}((\mathcal{M}_R^{sym})^c)[(W'_{\mathcal{M}^{sym}_R})^{-1}]
 \simeq
 \N_{\Delta}((\mathcal{M}_R^{sym})^c)[W_{\mathcal{M}^{sym}_R}^{-1}]$ of
 $\infty$-categories. 
\end{proof}

\begin{prop}\label{sm1}
Let $R$ be a cofibrant-fibrant commutative $S$-algebra. Let $Q$ be a cofibrant
 replacement functor on $\mathcal{M}_{\Phi(R)}^{sym}$ \cite[Proposition
 3.2.3.13]{HSS}. 

Let $\Lambda$ and $\Phi$ be the left and right Quillen adjoint functors given in Theorem~\ref{schw}. 
Let $R'$ be the commutative $\E$-ring which corresponds $Q \Phi(R) \in
 \mathcal{M}^{sym}$ under the equivalence in Proposition~\ref{smod3}. 

Then, there is an
equivalence of $\infty$-categories
\begin{equation}\label{siki}
 \N_{\Delta}(\mathcal{M}_{R}^{\circ})[W_{\mathcal{M}_R}^{-1}] \simeq \Mod_{R'}.
\end{equation}
\end{prop}
\begin{proof}
By Proposition~\ref{smod3} and 
Lemma~\ref{smi}, it is sufficient to show the equivalence  
$\N_{\Delta}((\mathcal{M}_{Q \Phi(R)}^{sym})^c)[W'_{\mathcal{M}_{Q \Phi(R)}^{sym}}] \simeq
 \N_{\Delta}(\mathcal{M}_R^{\circ})[W_{\mathcal{M}_R}]$, where $\mathcal{M}_{Q \Phi(R)}^{sym}$
is endowed with positive stable model structure. 

We regard $(\mathcal{M}_{Q \Phi(R)}^{sym})^c$ and $\mathcal{M}_R^{\circ}$ as ordinary
 categories with weak equivalences. 
$\Lambda$ and $\Phi$ the left and right Quillen adjoint functor given in
 Theorem~\ref{schw}.   

Since $\Lambda(\ast) = \ast$, $\Lambda$ preserves cofibrant objects. Since every
 object in $\mathcal{M}_S$ is fibrant, we take a functor $F:
 (\mathcal{M}_{Q \Phi(R)}^{sym})^c \to \mathcal{M}_R^{\circ}$ as the composition
 $(\mathcal{M}_{Q \Phi(R)}^{sym})^c \subset \mathcal{M}_{Q \Phi(R)}^{sym}$ with $\Lambda$. Since
 $\Lambda$ is a Quillen equivalence, $F$
 preserves weak equivalences on $(\mathcal{M}_{Q \Phi(R)}^{sym})^c$. We also take
 $G: \mathcal{M}_R^{\circ} \to (\mathcal{M}_{Q \Phi(R)}^{sym})^c$ as the
 composition $\mathcal{M}_R^{\circ} \subset \mathcal{M}_R$ with $Q \circ
 \Phi$. Since $\Phi$ preserves weak equivalences on fibrant objects, $G$
 preserves weak equivalences. Thus, we obtain the adjunction
 $\N_{\Delta}((\mathcal{M}_{Q \Phi(R)}^{sym})^{c}) \rightleftarrows
 \N_{\Delta}(\mathcal{M}_R^{\circ})$ of simplicial set marked by weak equivalences. Since
 $(\Lambda, \, \Phi)$ is a Quillen equivalence and we have a cartesian equivalence between $\N_{\Delta}(\mathcal{C})[W_{\N_{\Delta}(\mathcal{C})}^{-1}]^{\natural}$ and the marked simplicial set $(\N_{\Delta}(\mathcal{C}), W_{\N_{\Delta}(\mathcal{C})})$, this adjunction induces
 an equivalence $\N_{\Delta}((\mathcal{M}_{Q \Phi(R)}^{sym})^c)[(W'_{\mathcal{M}^{sym}_{Q \Phi(R)}})^{-1}] \simeq
 \N_{\Delta}(\mathcal{M}_R^{\circ})[W_{\mathcal{M}_R}^{-1}]$ of
 $\infty$-categories by \cite[Proposition 3.1.3.5 (2)]{HT}.   
\end{proof}

\section{Subcategories of perfect $R$-modules}
\begin{defn}\label{coherent}
A connective ring spectrum $R$ is {\it coherent} if $\pi_0 R$ is
coherent (i.e. every finitely generated ideal is finitely presented as
$\pi_0 R$-module) and $\pi_n R$ is finitely presented $\pi_0 R$-module
for $n \ge 0$. 
\end{defn}

\begin{defn}\label{pf}
Let $R$ be a connective $\E$-ring. 
\begin{enumerate}[(i)]
\item We say that $R$-module $M$ in $\Mod_R$ (resp. in
      $\mathcal{M}_R$) is a discrete $R$-module if its
      homotopy group $\pi_n M$ vanishes if $n$ is not equal to $0$. 
\item We say that $R$-module $M$ in $\Mod_R$ (resp. in
      $\mathcal{M}_R$) is Tor-amplitude $\le n$ if, for all $i > n$, $\pi_{i}(M \otimes_R N)=0$ for any discrete $R$-module $N$ (resp. any
      cofibrant discrete $R$-module $N$). 
\item For a coherent $\E$-ring $R$, we define an $\infty$-category $\Mod_R^{n, p}$ by a
      full $\infty$-subcategory of $\Mod_R^{perf}$
 consisting of the objects which is connective and have Tor-amplitude $\le n$.  
\item For a coherent ring spectrum $R$, a full subcategory $\mathcal{M}_R^{p} \subset
 \mathcal{M}_R$ is defined by those $R$-modules such that $\pi_n M = 0$ for
      sufficiently small $n$, $\pi_m M$ is finitely presented $\pi_0
      R$-modules for every $m \in \Z$ and there exists $n$ such that $M$ has
      Tor-amplitude $\le n$. 
\item For a coherent ring spectrum $R$, we define a category $\mathcal{M}_R^{n,p} \subset \mathcal{M}_R^{p}$ by a full subcategory of those connective $R$-modules of Tor-amplitude $\le n$ for fixed $n$.  
\end{enumerate}
\end{defn}
\begin{rem}
Note that, if $R$ is a connective coherent $\E$-ring, by \cite[Proposition
7.2.5.23 (4), Proposition
7.2.5.17]{HA}, the condition of perfect is
described by the condition on homotopy groups as Definition~\ref{pf}. 
\end{rem}

We denote by $(\mathcal{M}_R^{n,p})^{\circ}$ full subcategory of cofibrant-fibrant objects in $\mathcal{M}_R^{n,p}$ with respect to the model structure of
 $\mathcal{M}_R$. 

\begin{lemma}\label{Mcyl2}
\begin{enumerate}[(i)]
\item The subcategory $\mathcal{M}^{n,p}_R \subset \mathcal{M}_R$ is closed
 under weak equivalences. 
\item Mapping cylinders and mapping path spaces exist in $\mathcal{M}^{n,p}_R$. 
\end{enumerate}
\end{lemma}
\begin{proof}
The assertion (i) follows from the direct calculation of homotopy groups since the cofibrant replacement is a weak equivalence. 

Then, (ii) follows from Lemma~\ref{Mcyl}. ( We remark that the first assertion also follows from \cite[III, Theorem 3.8]{EKMM}.) 
\end{proof}

\begin{lemma}\label{n1}
Let us regard the categories $\mathcal{M}_R^{n,p}$ and $\mathcal{M}_R$
 as the relative categories with respect to the weak equivalences. We have an embedding $L^H (\mathcal{M}_R^{n,p})^{\circ} \subset L^H
 (\mathcal{M}_R)^{\circ}$ such that it induces the weak homotopy equivalence on
 mapping spaces. 
\end{lemma}

\begin{proof}
By the construction of hammock localization $L^H$,
 $\mathcal{M}_R^{n,p} \subset \mathcal{M}_R$ induces an embedding $L^H \mathcal{M}_R^{n,p} \subset L^H
 \mathcal{M}_R$. 
Since $\mathcal{M}_R^{n,p} \subset \mathcal{M}_R$ is a homotopically
 full subcategory by Lemma~\ref{Mcyl2}(i), it follows from
 Lemma~\ref{nn1} and the fact that $L^H (\mathcal{M}_R^{n,p})^{\circ}
 \simeq L^H \mathcal{M}_R^{n,p}$ and $L^H
 (\mathcal{M}_R)^{\circ} \simeq L^H
 \mathcal{M}_R$ by \cite[Proposition 5.2]{DK3}. 
\end{proof}

\begin{lemma}\label{k86}
We have
$\N_{\Delta}((\mathcal{M}_R^{n,p})^c)[W_{\mathcal{M}_R}^{-1}] \simeq \Mod_R^{n,p}$. 
\end{lemma}
\begin{proof}
By Lemma~\ref{Mcyl2} and Lemma~\ref{Mcyl}, the inclusion $\mathcal{M}_R^{n,p} \subset \mathcal{M}_R$ satisfies the assumption of Lemma~\ref{nn1}. 

By Lemma~\ref{n1}, we take the inclusion of $L^H
 (\mathcal{M}_R^{n,p})^{\circ} \subset L^H (\mathcal{M}_R)^{\circ}$ of
 simplicial categories, and we replace their
 mapping spaces by the associated simplicial sets defined as
 Definition~\ref{52}. 
In this proof, we denote them by $L (\mathcal{M}_R^{n,p})^{\circ} \subset L (\mathcal{M}_R)^{\circ}$. 

Note that $L (\mathcal{M}_R^{n,p})^{\circ}$ (resp. $L (\mathcal{M}_R)^{\circ}$) is Dwyer-Kan equivalent to
 its hammock localization by Lemma~\ref{nn1}. 
Therefore, $\N_{\Delta}(L (\mathcal{M}_R^{n,p})^{\circ}) \subset \N_{\Delta}(L (\mathcal{M}_R)^{\circ})$ become the strict
 model of $\N_{\Delta}((\mathcal{M}_R^{n,p})^{\circ})[W_{\mathcal{M}_R^{n,p}}^{-1}] \subset \N_{\Delta}(\mathcal{M}_R^{\circ})[W_{\mathcal{M}_R}^{-1}]$.

\textbf{Step (i)}
We consider the
 diagram of simplicial sets, 
\begin{equation}\label{e1} 
 \xymatrix@1{
  \N_{\Delta}((\mathcal{M}_R^{n,p})^{\circ})[W_{\mathcal{M}_R^{n,p}}^{-1}] \ar[r] \ar[d] & \N_{\Delta}(\mathcal{M}_R^{\circ})[W_{\mathcal{M}_R}^{-1}] \ar[d] \\
\N_{\Delta}((h\mathcal{M}_R^{n,p})^{\circ})  \ar[r] & \N_{\Delta}(h\mathcal{M}_R^{\circ}),
} 
\end{equation}
where the horizontal morphisms are induced by the inclusions of full
 subcategories. 

We already have the equivalence $\N_{\Delta}((\mathcal{M}_R^{n,p})^{\circ})[W_{\mathcal{M}_R}^{-1}] \simeq
 \N_{\Delta}((L^H \mathcal{M}_R^{n,p})^{fib})$. Note that the homotopy category of $L (\mathcal{M}_R^{n,p})^{\circ}$
 (resp. $L (\mathcal{M}_R)^{\circ})$ is a homotopy category $h \mathcal{M}_R^{n,p}$ (resp. $h\mathcal{M}_R$). 
Since $\N_{\Delta}$ is the right adjoint, to check the
 right vertical morphism is a fibration with respect to the Joyal model
 structure, it suffices to prove that the projection $L \mathcal{M}_R
 \to h\mathcal{M}_R$ is a fibration with respect to Bergner model
 structure, which follows from the axiom of model
 category.  

\textbf{Step (ii)}
According the definition of an $\infty$-subcategory~\cite[1.2.11]{HT}, we show that the diagram (\ref{e1}) is cartesian of simplicial sets.  
By Lemma~\ref{Mcyl2}(i) and Lemma~\ref{nn1}, it suffices to show that
 the cartesian for objects and morphisms in $\mathcal{M}_R$. (Note that
 the higher simplices of ordinary nerve is determined from $0$-simplices
 and $1$-simplices. )

Take $X \in \mathcal{M}_R^{\circ}$ which is isomorphic to $Y \in
 h(\mathcal{M}_R^{n,p})^c$. Then, there is an object $\tilde{Y} \in
 (\mathcal{M}_R^{n,p})$ which is weakly equivalent to $X$. Since
 $\mathcal{M}_R^{n,p}$ is closed under weak equivalences, we have $X \in
 \mathcal{M}_R^{n,p}$. 
For an arbitrary simplices, by Lemma~\ref{n1}, the upper horizontal morphism is inclusion of simplicial fullsubsets, and we have representatives
 of the simplices in $\N_{\Delta}((\mathcal{M}_R^{n,p})^{\circ})[W_{\mathcal{M}_R^{n,p}}^{-1}]$. Note that the
 composition law is well-defined 
by two out of three property.  
Thus, it is cartesian.  

\textbf{Step (iii)}
Next, according to the definition of an $\infty$-subcategory, we will consider the
 diagram of simplicial sets, 
\begin{equation}\label{e2} 
 \xymatrix@1{
  \Mod_R^{n,p} \ar[d] \ar[r] & \Mod_R \ar[d] \\
\N_{\Delta}((h\mathcal{M}_R^{n,p})^{\circ})  \ar[r] & \N_{\Delta}(h\mathcal{M}_R^{\circ}),
} 
\end{equation}
where the horizontal morphisms are the inclusions of full
 $\infty$-subcategories and the right vertical morphism is obtained by
 the identification of $h\mathcal{M}_R^c \cong h\Mod_R$ under
 (\ref{siki}) and the unit map for the adjoint functors
 \[
  h \circ \mathfrak{C} : \sSet \rightleftarrows Cat : \N_{\Delta} \circ
 i ,
 \]
where the functor $h$ is given by taking the homotopy category and $i$
 is the inclusion of the category $Cat$ of ordinary categories into the
 category $\sCat$ of simplicial categories.  
Note that the
 right vertical morphism is a fibration. 

Next step, by using that the objects in $\Mod_R^{n,p}$ is characterized
 by their homotopy groups, we will see that the right vertical morphism induces the left vertical
 morphism. Then, it automatically follows that (\ref{e2}) is cartesian. 

\textbf{Step (iv)}
Note that the category equivalence
 of stable homotopy categories preserves the smash products since the
 smash products on a stable homotopy category is determined up to
 isomorphisms. 

Take $\tilde{X} \in \Mod_R$. Assume that the image of $\tilde{X}$ in
 $h\mathcal{M}_R^c$ is in $h(\mathcal{M}_R^{n,p})$ under the right
 vertical morphism. 
For a discrete cofibrant-fibrant $R$-module $\tilde{N}$
 in $\Mod_R$, we have its image in $h(\mathcal{M}_R^{n,p})$. 
Then, the tensor product $\tilde{N} \otimes_R
 \tilde{X}$ in $\Mod_R$ is sent to the object in
 $h(\mathcal{M}_R^{n,p})$ under the right vertical morphism, 
and they have the same homotopy groups, so
 that we conclude that $\tilde{X} \in \Mod_R^{n,p}$. Since the upper
 horisontal morphism is inclusion of simplicial sets, this construction  
shows that the diagram (\ref{e2}) is cartesian.  

\textbf{Step (v)}
Since the weak equivalences on cofibrant-fibrant objects in a model category
 is the homotopy equivalences, e.g., invertible morphisms. Therefore, by
 recalling Definition~\ref{52}, all $1$-simplices in the mapping space 
of $L \mathcal{M}_R^{\circ}$ (resp. $L (\mathcal{M}_R^{n,p})^{\circ}$)
 is invertible, so that it is a Kan complex. Therefore, $L
 \mathcal{M}_R^{\circ}$ (resp. $L (\mathcal{M}_R^{n,p})^{\circ}$) is a fibrant object with respect to the Bergner
 model structure.

Since $\N_{\Delta}$ is the right adjoint, it preserves the fibrations
 and fibrant objects. By applying Coglueing lemma (cf. \cite{HT} A.2.4.3) to the diagrams (\ref{e1}) and (\ref{e2}), we have $\N_{\Delta}((\mathcal{M}_R^{n,p})^c)[W_{\mathcal{M}_R^{n,p}}^{-1}] \simeq \Mod_R^{n,p}$.   
\end{proof}

\section{The proof of $K(\Mod_R^{proj}) \simeq K(\Mod_R^{perf})$}
\subsection{Terminology of w-cofibrations and w-fibrations}
Let $\mathcal{C}$ be a pointed category. We fix a zero object of
 $\mathcal{C}$ and denote by $\ast$. 
Recall that we say that $\mathcal{C}$ is a
 Waldhausen category if it has two subcategories denoted by
 $co(\mathcal{C})$ and $W$, where morphisms in $co(\mathcal{C})$ are
 called w-cofibrations and the morphisms in
 $W$ is called weak equivalences, which satisfy the axiom of
 Waldhausen category in \cite{ww}. 

\begin{rem}
In terminology of \cite{ww}, a w-cofibration in a Waldhausen category is called a cofibration. Note that we use the term ``cofibration '' as a certain
 class of morphisms in a model category. 
\end{rem}

We define a class of morphisms, called the w-fibrations, on a pointed
category $\mathcal{C}$ with a zero object $\ast$ as follows : a class of {\it w-fibrations} is a class of morphisms in
$\mathcal{C}$ whose image in the opposite category $\mathcal{C}^{op}$
satisfies the axiom of a clas of w-cofibrations. 
We say that a diagram in a Waldhausen category $\mathcal{C}$ with weak equivalences $W$ is homotopy cocartesian if it gives a homotopy cocartesian in $L^H(\mathcal{C}, W)$. 

Let $\mathcal{C}$ be a pointed category endowed with w-fibrations. A map
$f: A \to B$ in $\mathcal{C}$ is said to be a weak w-fibration if it is
the composition of a w-fibration with a zig-zag of weak equivalences. It
is the dual notion of weak w-cofibrations defined in \cite[Definition 2.2]{BM}.

We say that a pointed category $\mathcal{C}$ defined above admits the functorial factorization of w-fibrations
if any weak w-fibration is factored functorially as a weak equivalence
      followed by a w-fibration in $\mathcal{C}$. We call the following condition {\it
 saturated} : a morphism is a weak equivalence if and only if it is an
 isomorphism in the homotopy category~\cite[Theorem 6.4]{BM1}. Note that a model category is saturated.

%
%
%
%
%
%

\subsection{The proof of $K(\Mod_R^{proj}) \simeq K(\Mod_R^{perf})$}
Now, recall that the notion of cofibrants and fibrants in $\mathcal{M}_R$
from Remark~\ref{ms1}. 

We show the several properties of $(\mathcal{M}_R^{n,p})^{\circ}$
(resp. $\mathcal{M}_R^{n,p}$). 
\begin{lemma}\label{ky3}
Let $M' \to M \to  M''$ be a fiber sequence of $R$-modules in
 $(\mathcal{M}_R^p)^{\circ}$ (resp. $\mathcal{M}_R^{n,p}$).
\begin{enumerate}[(i)]
\item Assume that $M'$ and $M''$ have Tor-amplitude $\le n$, and $M$ has Tor-amplitude $\le n-1$. Then, $M'$ has
 Tor-amplitude $\le n-1$. 
\item Assume that $M'$ and $M''$ have Tor-amplitude $\le
 n$. Then, $M$ has Tor-amplitude $\le n$. 
\item $(\mathcal{M}_R^{n,p})^{\circ}$ (resp. $\mathcal{M}_R^{n,p}$) is closed under extension, has a
      direct sum. 
\end{enumerate}
\end{lemma}
\begin{proof}
Let $N$ be a cofibrant discrete $R$-module. We prove that $\pi_k(N
\otimes_R M') \simeq 0$ for $k \ge n$. 
We have an exact sequence of homotopy groups
\[
 \pi_{k+1}(N \otimes_R M'') \to \pi_{k}(N \otimes_R M') \to \pi_{k}(N \otimes_R M).
\]
We prove the assertion $(i)$. If $k \ge n $, $\pi_{k+1}(N \otimes_R M'')$ and $\pi_{k}(N \otimes_R M)$ vanish by assumption that $M''$ has Tor-amplitude $\le n$
 and $M$ has Tor-amplitude $\le n-1$.  

If $k \ge n+1$, $\pi_{k}(N \otimes_R M')$ and $\pi_{k}(N \otimes_R M'')$
 vanish in the above exact sequence of homotopy
 groups. Therefore the assertion $(ii)$ is proved.   

For (iii), the assertion $(ii)$ shows $(\mathcal{M}_R^{n,p})^{\circ}$ (resp. $\mathcal{M}_R^{n,p}$) is
closed under extension. Since $(\mathcal{M}_R^{n,p})^{\circ}$ (resp. $\mathcal{M}_R^{n,p}$) is a full
subcategory of the stable category $\mathcal{M}_R$, its
homotopy category is additive category.  
\end{proof}


Note that the forgetfull functor from $\mathcal{M}_R$ to the category of
$S$-modules preserves the
Hurewicz cofibrations~\cite[Theorem 12.1]{MMSS} and the extention of
coefficients of modules preserves the Hurewicz cofibrations~\cite[Lemma
12.2]{MMSS}.

\begin{defn}\label{hhh}
We define a w-cofibration (resp. a w-fibration) in $\mathcal{M}_R^{n,p}$ as follows. 
\begin{enumerate}[(i)]
\item We define a w-cofibration $X \to Y$ in
      $(\mathcal{M}_R^{n,p})^{\circ}$ if it is a Hurewicz
      cofibration in $\mathcal{M}_R^{\circ}$ and
its cofiber lies in $(\mathcal{M}_R^{n,p})^{\circ}$. 
\item We also define a w-fibration $Y \to Z$ in $\mathcal{M}_R^{n,p}$ if it is a fibration in $\mathcal{M}_R$ and its
fiber lies in $\mathcal{M}_R^{n,p}$. 
\end{enumerate}
Then, $(\mathcal{M}_R^{n,p})^{\circ}$ is a Waldhausen
 category with the w-cofibrations in $\mathcal{M}_R^{n,p}$ 
 and $(\mathcal{M}_R^{n,p})^{op}$ is a Waldhausen category
 with the w-fibrations in $\mathcal{M}_R^{n,p}$. 
\end{defn}
Note that, once we define the w-cofibrations and w-fibrations of
$\mathcal{M}_R^{n,p}$, the weak w-cofibrations and weak
w-fibrations in $\mathcal{M}_R^{n,p}$ are automatically
determined.

\subsection{Resolution theorem for $\mathcal{M}_R^{n,p}
\subset \mathcal{M}_R^{n+1,p}$}
\begin{defn}\label{TA}
Let $\mathcal{C}$ be a full subcategory of a pointed model category. Let
 us take a cofibration between cofibrants such that its cofiber lies in
 $\mathcal{C}^c$ (resp. fibration between fibrants such that its fiber lies in
 $\mathcal{C}^f$) as a w-cofibration (resp. a w-fibration). We say that $\mathcal{C}$
 satisfies the assumption (A) if it satisfies the following conditions:
\begin{enumerate}[(i)]
\item $\mathcal{C}$ is closed under extensions.
\item The w-cofibrations make $\mathcal{C}^c$ into a Waldhausen category.
\item The category $(\mathcal{C}^f)^{op}$ becomes a Waldhausen
 category by the w-fibrations in $\mathcal{C}$. 
\item $\mathcal{C}$ is saturated 
\item $\mathcal{C}^c$ (resp. $\mathcal{C}^f$) has functorial
      factorizations of w-cofibrations (resp. w-fibrations) respectively. 
\item Let $W$ denote the category of weak equivalences in
      $\mathcal{C}$. Then, the homotopy category $hL^H(\mathcal{C}, W)$ is additive. 
\end{enumerate}
\end{defn}

%
%

\begin{lemma}\label{TA0}
The category $\mathcal{M}_R^{n,p}$ satisfies the
 assumption (A) in Definition~\ref{TA}. 
\end{lemma}
\begin{proof}
We check the conditions of (A) in Definition~\ref{TA}. 
For (i) of (A), it follows from Lemma~\ref{ky3}(i). 

The assertion (ii), (iii) follows from Definition~\ref{hhh}. 

Since $(\mathcal{M}_R^{n,p})^{\circ}$ is the full subcategory of the
 model category $\mathcal{M}_R$, (iv) is automatically
 satisfied. 
The assertion (v) of (A) follows from Lemma~\ref{Mcyl2}. 
For (vi) of (A), it follows from Lemma~\ref{ky3}(iii). 
\end{proof}

\begin{rem}[cf. \cite{ba} Section 5.10]\label{mochr}
Let $\mathcal{C}$ be a Waldhausen category that admits functorial factorization of w-cofibrations and saturated.
We define $wS^W_n \mathcal{C}$ (resp. $wS'_n \mathcal{C}$) by the category
of weak equivalences in $S^W_n \mathcal{C}$ (resp. by the nerve of the
category of weak equivalences in $S'_n \mathcal{C}$). 
We denote by $wS^W_{\bullet} \mathcal{C}$ (resp. $wS'_{\bullet}
\mathcal{C}$) the bisimplicial set which sends $[n] \in \Delta^{op}$ to $wS^W_n \mathcal{C}$ (resp. $wS'_n \mathcal{C}$). 
Then, the inclusion $wS^W_{\bullet} \mathcal{C} \to wS'_{\bullet}\mathcal{C}$
 induces a weak equivalence $wS^W_{\bullet}\mathcal{C} \to
 wS'_{\bullet}\mathcal{C}$ of bisimplicial set~\cite[Theorem 2.9]{BM}. 

We remark that the following assertion is well-known in the $K$-theory. 
Let $\mathcal{C}$ be a full subcategory of a pointed model category 
with the assumption (A) in Definition~\ref{TA}. 
Let us regard $(\mathcal{C}^f)^{op}$ as a Waldhausen
 category with the w-fibrations in $\mathcal{C}$. 
Then, we have a functorial equivalence $K(\mathcal{C}^c) \simeq K((\mathcal{C}^f)^{op})$. 
\end{rem}

We will apply the following resolution theorem to $\mathcal{M}_R^{n,p}
\subset \mathcal{M}_R^{n+1,p}$. By the duality of Remark~\ref{mochr}, we
state the following resolution theorem, which is due to Mochizuki, by
the term of w-fibrations.

\begin{thm}[cf. \cite{Moc1} Theorem 1.13]\label{moc}
Let $\mathcal{A} \subset
 \mathcal{B}$ be an inclusion of full subcategories of a pointed model
 categories.   
Assume that $\mathcal{A}$ and $\mathcal{B}$ satisfy the assumption (A)
 in Definition~\ref{TA}. 

Let $\mathcal{A}(m, w)$
 $($resp. $\mathcal{B}(m, w)$ $)$ be a full subcategory of the functor
 category $\Fun([m], \, \mathcal{A})$ $($resp. $\Fun([m], \,
 \mathcal{B})$ $)$ which consists of the functors taking values in the
 category of weak equivalences $w\mathcal{A}$ in $\mathcal{A}$
 $($resp. $w\mathcal{B}$ in $\mathcal{B}$ $)$. Assume that, for each $m \ge
 0$, $\mathcal{A}(m, w) \subset \mathcal{B}(m, w)$ satisfies the
 following conditions called the resolution condition:
\begin{enumerate}[(i)]
\item Closed under extensions,
\item For any $B \in \mathcal{B}(m, w)$, there exists $A \in
      \mathcal{A}(m, w)$ and a w-fibration $A \to B$,
\item For any fiberation sequence $A' \to A \to B$ in $\mathcal{B}(m,
      w)$, 
$A'$ is an object in $\mathcal{A}(m, w)$ if $A \in
      \mathcal{A}(m, w)$.  
\end{enumerate}
Then, $\mathcal{A} \subset \mathcal{B}$ induces an equivalence
 $K(\mathcal{A}^c) \simeq K(\mathcal{B}^c)$ for Waldhausen categories
 $\mathcal{A}^c$ and $\mathcal{B}^c$ obtained by (A)-(ii). 
\qed
\end{thm}

\begin{prop}\label{fact}
Let $\mathcal{A} \subset
 \mathcal{B} \subset \mathcal{M}_R$ be an inclusion of full
 subcategories of a pointed model categories such that $\mathcal{A}$ and $\mathcal{B}$ satisfy the assumption (A)
 in Definition~\ref{TA}. 

Assume that if $A \in \mathcal{A}$
 is weakly equivalent to a cofibrant object $M$ in $\mathcal{M}_R$, $M$ is
 the object of $\mathcal{A}$.  
Then, if $\mathcal{A} \subset \mathcal{B}$ satisfies the
 resolution condition in Theorem~\ref{moc} for $m = 0$,
 $\mathcal{A} \subset \mathcal{B}$ satisfies the assumption of
 Theorem~\ref{moc}. 
\end{prop}
\begin{proof}
Take the inclusion $\mathcal{A}(m, w) \to \mathcal{B}(m, w)$ and the fiber sequence $x \to y \to z$ in $\mathcal{B}(m, w)$. 
Assume that, if $m = 0$, the conditions are satisfied. 
If $x$ and $z$ are objects in $\mathcal{A}(m, w)$, for each $0 \le i \le
 m$, $x_i$ and $z_i$ are in $\mathcal{A}$. By the assumption,
 $y_i$ is in $\mathcal{A}$ for each $i$. Thus, the condition $(i)$ in Theorem~\ref{moc} is obvious. The condition
 $(iii)$ follows by the same
 argument. We will check the
 condition in Theorem~\ref{moc}(ii) by induction. 

We proceed by induction on $m$. 
The condition $(ii)$ is valid for $m=0$. 
For $m \ge 0$, an object in $z \in \mathcal{B}(m, w)$ is written in the
 form of a sequence of weak equivalences $z_{0} \to \cdots, \to
 z_{m}$ and a morphism between
 them is given by a diagram. 

Take $z' = (z_0 \to \cdots, \to z_{m-1}) \in \mathcal{B}(m-1, w)$. 
By the induction hypothesis, we have a fibration $y' \to z'$ in $\mathcal{B}(m-1, w)$
 represented by the following diagram:
\[ 
 \xymatrix@1{
  y_0 \ar[r]\ar[d] & \cdots \ar[r]\ar[d] & y_{m-1} \ar[d] \\
z_0 \ar[r] &  \cdots \ar[r] & z_{m-1} \ar[r] & z_m,
} 
\] 
where the vertical morphisms are fibration. 
For the composition $y_{m-1} \to z_{m-1} \to z_m$, we apply the
 factorization of w-fibrations. Then, there is an object $y_m$ such that $y_{m-1}
 \simeq y_m$ and $y_m \to z_m$ is a fibration. By the assumption (ii), we have $y_m \in \mathcal{A}$. 
Thus, we can proceed the induction. 
\end{proof}

\begin{prop}\label{res0}
Let $R$ be a connective coherent $\E$-ring. 
Let $\mathcal{M}_R^{n,p} \to \mathcal{M}_R^{n+1,p}$ be the inclusion of $\infty$-categories. 
Then, the induced morphism $K((\mathcal{M}_R^{n,p})^{\circ})
 \simeq K((\mathcal{M}_R^{n+1,p})^{\circ})$ is an equivalence. 
\end{prop}
\begin{proof}
The category $\mathcal{M}_R^{n,p}$ satisfies the assumptions (i) and (ii) of Proposition~\ref{fact}. Therefore, by Proposition~\ref{fact}, it suffices to
 check that the inclusion $\mathcal{M}_R^{n,p} \to \mathcal{M}_R^{n+1, p}$ satisfies the condition in
 Theorem~\ref{moc} for $m=0$. The inclusion satisfies the condition
 (i) and (iii) for $m=0$ by Lemma~\ref{ky3}(i) and (iii). 
Take $M \in \mathcal{M}_R^{n+1,p}$. 

For an $S$-algebra $R$, let $\mathbb{F}_R S$ denote a sphere $R$-module $R \wedge_{\mathcal{L}} \mathbb{L}S$~\cite[III, Proposition
1.3]{EKMM}. Note that, for an object $A$ in $\mathcal{M}_R$, the $n$-th
homotopy group of $A$ is defined by the hom-set
$\Hom_{h\mathcal{M}_R}(\mathbb{F}_R S^n, \, A)$, where $\mathbb{F}_R
S^n$ is the $n$-times shift of the sphere $R$-module. 

Since $\pi_0M$ is not empty, 
we choose a morphism $\mathbb{F}_R S \to
 M$ corresponding to an object of $\pi_0 M$,  
where $\mathbb{F}_R S$ is sphere
 $R$-module. 
Note that $\mathbb{F}_R S$ is an object of $\mathcal{M}_R^{n,p}$ for every
 $n \ge 0$, and the fiber of a morphism $\mathbb{F}_R S \to
 M$ is obviously connective. 
The condition $(ii)$
 is satisfied by a fibrant replacement of $\mathbb{F}_R S \to M$ since
 $\mathcal{M}_R^{n,p}$ is closed under equivalences for every
 $n \ge 0$. 
Thus, we have an equivalence $K((\mathcal{M}_R^{n,p})^{\circ})
 \simeq K((\mathcal{M}_R^{n+1,p})^{\circ})$. 
\end{proof}

\subsection{Comparison between the $K$-theory of $(\mathcal{M}_R^{n,p})^{\circ}$ and $\Mod_R^{n,p}$}

Let $\mathcal{C}$ be a pointed $\infty$-category with
$w^{\infty}$-cofibrations. 
Let $K(\mathcal{C})$ be the algebraic $K$-theory of $\infty$-category in
the sense of \cite[Section 2.5]{BGT-End}. 
We obtain the following
comparison of algebraic $K$-theory spaces, which is functorial with
respect to Waldhausen categories. 
Note that Barwick's construction of $K$-theory (cf. \cite[Theorem 7.6, Section 10]{ba}) is equivalent to Lurie's construction. 

\begin{thm}[\cite{BGT} Corollary 7.12 
]\label{S2}
Let $\mathcal{C}_0$ be a Waldhausen category which admits functorial
 factorization of w-cofibrations and is saturated. Let $W \subset
 \mathcal{C}_0$ be the category of weak equivalences. 

Then there is an zigzag of equivalences
\[
 K((\mathcal{C}_0, W)) \simeq K(N_{\Delta}(\mathcal{C}_0)[W^{-1}]),
\]
where the left hand side is the Waldhausen $K$-theory and the right hand side is the algebraic
 $K$-theory of $\infty$-category. 
\end{thm}

Now, we define a $w^{\infty}$-cofibration $X \to Y$ of $\Mod_R^{n,p}$ if it is a $w^{\infty}$-cofibration in $\Mod_R$ and
its cofiber lies in $\Mod_R^{n,p}$. 

Note that we have $\Map_{\Mod_R^{n,p}}(A, \, B) =
\Map_{\mathcal{M}_R^{n,p}}(A, \, B)$ for every cofibrant-fibrant objects
$A$ and $B$ in $\mathcal{M}_R^{n,p}$. 

\begin{lemma}\label{MM1}
Let us regard $(\mathcal{M}_R^{n,p})^{\circ}$ as the Waldhausen category
 with w-cofibrations obtained by Definition~\ref{hhh}. 
Then, the Waldhausen $K$-theory $K((\mathcal{M}_R^{n,p})^{\circ})$ is
 equivalent to the algebraic $K$-theory $K(\Mod_R^{n,p})$. 
\end{lemma}
\begin{proof}
The equivalence in Lemma~\ref{k86} is a weak exact functor since a
 diagram in $(\mathcal{M}_R^{n,p})^{\circ}$ is homotopy cocartesian if
 and only if its image in $\Mod_R^{n,p}$ under the equivalence is a
 pullback.  
Since the category $(\mathcal{M}_R^{n,p})^{\circ}$ admits a functorial
 factorizations of w-cofibrations, the assertion follows from Theorem~\ref{S2}. 
\end{proof}

\begin{prop}\label{res}
Let $R$ be a connective coherent $\E$-ring. Let  $\Mod_R^{n,p} \subset \Mod_R^{n+1,p}$ be the inclusion of $\infty$-categories. 
Then, the induced morphism $K(\Mod_R^{n,p}) \to K(\Mod_R^{n+1,p})$ is an equivalence. 
\end{prop}
\begin{proof}
Since we have $K(\Mod_R^{n,p}) \simeq K((\mathcal{M}_R^{n,p})^c)$ by Lemma~\ref{MM1}, the
 assertion follows from Lemma~\ref{res0}.   
\end{proof}

\begin{prop}\label{perfproj}
For a connective coherent $\E$-ring $R$, 
$K(\Mod_R^{proj}) \simeq K(\Mod_R^{perf})$.
\end{prop}
\begin{proof}
Note that $\Mod_R^{0,p} \simeq \Mod_R^{proj}$ by \cite[Proposition
 7.2.5.20, Remark 7.2.5.22]{HA}. 
Let us denote the $\infty$-category of connective perfect $R$-modules by
 $(\Mod_R^{perf})^{cn}$. 
Then, we have an equivalence $\mathrm{colim}_n
 \Mod_R^{n,p} \simeq (\Mod_R^{perf})^{cn}$ from Definition~\ref{pf}. Since the $K$-theory commutes with filterd colimits by \cite[Section 7]{ba1}, we have
 $K(\Mod_R^{proj}) \simeq K((\Mod_R^{perf})^{cn})$. 
Thus, it suffices to show that $K((\Mod_R^{perf})^{cn}) \simeq K(\Mod_R^{perf})$
. 

Consider the following colimit
\[ 
 \xymatrix@1{
 (\Mod_R^{perf})^{cn} \xrightarrow{\Sigma} &  \cdots \xrightarrow{\Sigma} &
 (\Mod_R^{perf})^{cn} \xrightarrow{\Sigma} & \cdots. 
} 
\] 
By \cite[Proposition 4.4]{ba1}, this filterd colimit exists as an
 $\infty$-category with w-cofibrations. 
From this filtered colimit, $\mathrm{colim}_{\Sigma}K((\Mod_R^{perf})^{cn}) \simeq K(\mathrm{colim}_{\Sigma}(\Mod_R^{perf})^{cn})$. 

We will show that the following equivalences 
\[
 \mathrm{colim}_{\Sigma}K((\Mod_R^{perf})^{cn}) \simeq
 K((\Mod_R^{perf})^{cn})
\] 
and
\[
 K(\mathrm{colim}_{\Sigma}(\Mod_R^{perf})^{cn}) \simeq
 K(\Mod_R^{perf}). 
\]
Since we have the following cofiber sequence in $(\Mod_R^{perf})^{cn}$
\[ 
 \xymatrix@1{
 id \ar[r] \ar[d] & 0 \ar[d] \\
 0 \ar[r] & \Sigma, 
} 
\]
and $\Sigma$ induces $-id$ on $K$-theory, $\mathrm{colim}_{\Sigma}K((\Mod_R^{perf})^{cn})$ is equivalent to $K((\Mod_R^{perf})^{cn})$. We show that
 $\mathrm{colim}_{\Sigma}(\Mod_R^{perf})^{cn} \simeq \Mod_R^{perf}$. It
 suffices to show that $\mathrm{colim}_{\Sigma}(\Mod_R^{perf})^{cn}$ is
 a stable $\infty$-category. Indeed, it has cofibers, and the
 endofunctor $\Sigma$ is an equivalence. By \cite[Lemma 1.1.3.3]{HA},
 $\mathrm{colim}_{\Sigma}(\Mod_R^{perf})^{cn}$ is stable.    
\end{proof}

\section{Proof of Theorem~\ref{main2}}

We say that an $R$-module $M$ is {\it truncated} if $\pi_n M = 0$ for
sufficiently large $n$~\cite[Definition 5.5.6.1]{HT}. 
Let $R$ be a coherent $\E$-ring, which we recall in
Definition~\ref{coherent}. An $R$-module $M$ is {\it coherent}
$R$-module if it is truncated, $\pi_0 M = 0$ for sufficiently small $n$ and $\pi_n M$ is finitely presented $\pi_0 R$-module
for $n \ge 0$. 


\begin{defn}[\cite{ba} Definition 8.4, \cite{BL} Proposition 1.3]\label{regreg}
Let $R$ be a coherent $\E$-ring defined in Definition~\ref{coherent}. 
\begin{enumerate}[(i)]
\item $R$ is {\it almost regular} if any coherent
      $R$-module has Tor-amplitude $\le n$ for some $n \in
      \Z_{\ge 0}$. 
\item $R$ is {\it regular} if $\pi_0 R$ is regular and $R$ is
almost regular. 
\end{enumerate}
\end{defn}

\begin{defn}
Let $R$ be a coherent $\E$-ring. 
\begin{enumerate}[(i)]
\item We say that $M$ is almost
perfect if $\pi_m M = 0$ for sufficiently small $m$ and $\pi_n M$ is
finitely presented over $\pi_0 R$ for $n \in \Z$ ~\cite[Proposition 7.2.5.17]{HA}. 
\item Let $\Coh_R$ be a full $\infty$-subcategory
of $\Mod_R$, which consists of almost perfect and truncated objects. It
      is a stable $\infty$-category. 
\end{enumerate}
\end{defn}

\begin{lemma}\label{key2}
Let $R$ be a connective $\E$-ring. We define $(\Mod_{R}^{proj})^b$ by an
$\infty$-category of finitely generated projective truncated
$R$-modules.  We define $(\Mod_R^{perf})^b$ by an $\infty$-category of
perfect truncated $R$-modules. 
\begin{enumerate}[(i)]
\item If a connective $\E$-ring $R$ has only finitely many non-zero
      homotopy groups, any finitely generated projective $R$-module has only finitely many non-zero homotopy groups. 
\item Suppose that $R$ is an almost regular $\E$-ring. Then $(\Mod_R^{perf})^b$ coincides with $\Coh_R$.  
\item Let $R$ be a connective $\E$-ring with only finitely many non-zero
 homotopy groups. Then, the natural inclusion $(\Mod_{R}^{perf})^b \to \Mod_{R}^{perf}$ is an equivalence. 
\end{enumerate}
\end{lemma}
\begin{proof}
If $R$ has only
 finitely many non-zero homotopy groups, so is a finite copies of $R$. 
Since a finitely generated projective $R$-module is
 a retract of direct summand of finite copies of $R$, the assertion
 $(i)$ holds. 
The assertion $(ii)$ is the consequence of \cite[Proposition 8.6]{ba}. 
To show $(iii)$, it suffices to show that the perfect $R$-module $M$ is
 truncated. 
By applying \cite[Proposition 2.13 (7)]{AG}
 inductively, and by \cite[Proposition 2.13 (6)]{AG}, it comes down to
 the case of a shift of finitely generated projective $R$-module $\Sigma^a P$. It is truncated. 
Therefore if $R$ is truncated, so is $M$. 
\end{proof}
\begin{thm}\label{m2}
Let $R$ be a connective regular $\E$-ring with only finitely many non-zero
 homotopy groups. 
\begin{enumerate}[(i)]
\item The inclusion $\Mod_{R}^{proj} \to \Mod_{R}^{perf}$ induces an equivalence of $K$-theory spaces;
\[
K(\Mod_{R}^{proj}) \simeq K(\Mod_{R}^{perf}) \simeq K(\Coh_R). 
\]
\item Let $\mathcal{P}_{\pi_0 R}$ be an ordinary category of
 finitely generated projective $\pi_0 R$-modules. 
Then, $K(\Mod_{R}^{proj}) \simeq K(\mathcal{P}_{\pi_0
 R})$. 
\end{enumerate}
\end{thm}
\begin{proof}
The first part of the assertion follows from Proposition~\ref{perfproj} and
 Lemma~\ref{key2}. 
By the first part of the assertion and the main theorem of \cite{BL} with the
 regularity of $R$, we obtain $K(\Mod_{R}^{proj}) \simeq K(\mathcal{P}_{\pi_0 R})$. 
\end{proof}

\bibliographystyle{amsplain} \ifx\undefined\bysame
\newcommand{\bysame}{\leavemode\hbox to3em{\hrulefill}\,} \fi

\end{document}